\documentclass[leqno]{amsart}
\usepackage{caption}
\usepackage{amsmath,amssymb,amsthm,mathtools,mathrsfs,tikz}
\usepackage{environ}
\usepackage{stmaryrd}
\usetikzlibrary{arrows}
\usetikzlibrary{positioning}
\usetikzlibrary{decorations.text}
\usetikzlibrary{decorations.pathmorphing}
\usetikzlibrary{decorations.pathreplacing}
\makeatletter
\newsavebox{\@brx}
\newcommand{\llangle}[1][]{\savebox{\@brx}{\(\m@th{#1\langle}\)}%
	\mathopen{\copy\@brx\kern-0.5\wd\@brx\usebox{\@brx}}}
\newcommand{\rrangle}[1][]{\savebox{\@brx}{\(\m@th{#1\rangle}\)}%
	\mathclose{\copy\@brx\kern-0.5\wd\@brx\usebox{\@brx}}}
\makeatother
\makeatletter
\newsavebox{\measure@tikzpicture}
\NewEnviron{scaletikzpicturetowidth}[1]{%
	\mathrm{Def}\tikz@width{#1}%
	\mathrm{Def}\tikzscale{1}\begin{lrbox}{\measure@tikzpicture}%
		\BODY
	\end{lrbox}%
	\pgfmathparse{#1/\wd\measure@tikzpicture}%
	\edef\tikzscale{\pgfmathresult}%
	\BODY
}
\makeatother
\usepackage{bbm}
\DeclarePairedDelimiter\norm{\lvert}{\rvert}
\DeclarePairedDelimiter\inner{\langle}{\rangle}

\makeatletter
\let\save@mathaccent\mathaccent
\newcommand*\if@single[3]{%
	\setbox0\hbox{${\mathaccent"0362{#1}}^H$}%
	\setbox2\hbox{${\mathaccent"0362{\kern0pt#1}}^H$}%
	\ifdim\ht0=\ht2 #3\else #2\fi
}
\newcommand*\rel@kern[1]{\kern#1\dimexpr\macc@kerna}
\newcommand*\widebar[1]{\@ifnextchar^{{\wide@bar{#1}{0}}}{\wide@bar{#1}{1}}}
\newcommand*\wide@bar[2]{\if@single{#1}{\wide@bar@{#1}{#2}{1}}{\wide@bar@{#1}{#2}{2}}}
\newcommand*\wide@bar@[3]{%
	\begingroup
	\def\mathaccent##1##2{%
		\let\mathaccent\save@mathaccent
		\if#32 \let\macc@nucleus\first@char \fi
		\setbox\z@\hbox{$\macc@style{\macc@nucleus}_{}$}%
		\setbox\tw@\hbox{$\macc@style{\macc@nucleus}{}_{}$}%
		\dimen@\wd\tw@
		\advance\dimen@-\wd\z@
		\divide\dimen@ 3
		\@tempdima\wd\tw@
		\advance\@tempdima-\scriptspace
		\divide\@tempdima 10
		\advance\dimen@-\@tempdima
		\ifdim\dimen@>\z@ \dimen@0pt\fi
		\rel@kern{0.6}\kern-\dimen@
		\if#31
		\overline{\rel@kern{-0.6}\kern\dimen@\macc@nucleus\rel@kern{0.4}\kern\dimen@}%
		\advance\dimen@0.4\dimexpr\macc@kerna
		\let\final@kern#2%
		\ifdim\dimen@<\z@ \let\final@kern1\fi
		\if\final@kern1 \kern-\dimen@\fi
		\else
		\overline{\rel@kern{-0.6}\kern\dimen@#1}%
		\fi
	}%
	\macc@depth\@ne
	\let\math@bgroup\@empty \let\math@egroup\macc@set@skewchar
	\mathsurround\z@ \frozen@everymath{\mathgroup\macc@group\relax}%
	\macc@set@skewchar\relax
	\let\mathaccentV\macc@nested@a
	\if#31
	\macc@nested@a\relax111{#1}%
	\else
	\def\gobble@till@marker##1\endmarker{}%
	\futurelet\first@char\gobble@till@marker#1\endmarker
	\ifcat\noexpand\first@char A\else
	\def\first@char{}%
	\fi
	\macc@nested@a\relax111{\first@char}%
	\fi
	\endgroup
}
\makeatother

\usepackage{bm}
\usepackage[colorlinks=true,linkcolor=blue,citecolor=blue]{hyperref}
\usepackage{enumitem}
\usepackage{csquotes}

\def\irr#1{{\rm  Irr}(#1)}
\def\cd#1{{\rm  cd}(#1)}
\def\cl{{\rm cl}}

\newcounter{intro}

\newtheorem{introthm}[intro]{Theorem}

\newtheorem{thm}{Theorem}[section]
\newtheorem{lem}[thm]{Lemma}
\newtheorem{cor}[thm]{Corollary}

\theoremstyle{remark}
\newtheorem{rem}[thm]{Remark}
\theoremstyle{definition}

\newtheorem{exam}[thm]{Example}

\title{Partial GVZ-groups}

\author{Shawn T. Burkett}
\address{Department of Mathematical Sciences, Kent State University, Kent,
	Ohio 44242, U.S.A.} \email{sburkett@math.kent.edu}

\author{Mark L. Lewis}
\address{Department of Mathematical Sciences, Kent State University, Kent,
	Ohio 44242, U.S.A.} \email{lewis@math.kent.edu}

\date{\today}

\begin{document}

\begin{abstract}
Following the literature, a group $G$ is called a group of central type if $G$ has an irreducible character that vanishes on $G\setminus Z(G)$. Motivated by this definition, we say that a character $\chi\in\irr{G}$ has central type if $\chi$ vanishes on $G\setminus Z(\chi)$, where $Z(\chi)$ is the center of $\chi$. Groups where every irreducible character has central type have been studied previously under the name GVZ-groups (and several other names) in the literature. In this paper, we study the groups $G$ that possess a nontrivial, normal subgroup $N$ such that every character of $G$ either contains $N$ in its kernel or has central type. The structure of these groups is surprisingly limited and has many aspects in common with both central type groups and GVZ-groups. 
\end{abstract}

\maketitle

\section{Introduction}

Throughout this paper, all groups are finite.  For a group $G$, we write $\irr G$ for the set of irreducible characters of $G$.  Let $N$ be a normal subgroup of $G$. Take $\chi$ to be a character in $\irr G$ and assume that $\chi_N$ is homogeneous; i.e. that $\chi_N$ is a multiple of an irreducible character of $N$. Following the literature, when $\chi$ satisfies the condition that $\chi(g) = 0$ for every element $g\in G\setminus N$, we say $\chi$ is {\it fully ramified} over $N$.  It is well-known that $\chi$ being fully ramified over $N$ is equivalent to $\chi_N$ having a unique irreducible constituent $\theta$ and $\theta^G$ having a unique irreducible constituent.  This is also equivalent to $\chi_N$ having a unique irreducible constituent $\theta$ with the property that $|G:N| = (\chi (1)/\theta (1))^2$.  (See Lemma 2.29 and Problem 6.3 of \cite{MI76}.)

Following the literature, a group $G$ is called {\it central type} if there is an irreducible character of $G$ that is fully ramified over the center $Z(G)$.  Central type groups have been studied extensively in the literature.  We mention a few of the important papers: \cite{DJ}, \cite{espuelas}, \cite{gagola}, and \cite{IH}.  In particular, it is known that if $G$ is of central type, then $G$ is solvable (see \cite{IH}), and $G$ is central type if and only if all of the Sylow subgroups of $G$ are central type (see \cite{DJ}). 
 
With this as motivation, we say that an irreducible character $\chi$ of $G$ has {\it central type} if $\chi$, considered as a character of $G/\ker(\chi)$, is fully ramified over $Z(G/\ker(\chi))$.  (I.e., $G/\ker{\chi}$ is a group of central type with faithful character $\chi$.)  Groups where every member of $\irr{G}$ has central type have been called {\it GVZ-groups} in the literature (see \cite{ML19gvz, AN12gvz, AN16gvz}), and we use this terminology in this paper. Such groups are necessarily nilpotent (see \cite{AN12gvz}); thus, it is not difficult to see that $G$ is a GVZ-group if and only if $G$ is the direct product of GVZ groups of prime power orders for different primes.  Moreover, we show in \cite{ourpre} that $G$ is a GVZ-group if and only if $\cl_G(g)=g[g,G]$ for each element $g\in G$.

We have seen that having a central type character whose center equals the center of the group implies that the group is solvable and that having all irreducible characters be central type implies that the group is nilpotent.  It makes sense to ask what can be said if some but not necessarily all of the irreducible characters are known to have central type. 

When $N$ is a normal subgroup of a group $G$, we write $\irr {G \mid N}$ for the set of irreducible characters of $G$ whose kernels do not contain $N$.  This set was first studied explicitly by Isaacs and Knutson in \cite{Knutson} where they consider the impact of a number of conditions on $\irr {G \mid N}$ on the structure of $N$.  We now consider the assumption that all the characters in $\irr {G \mid N}$ have central type.  In particular, we obtain results regarding the structure of $G$ as opposed to just the structure of $N$.  

\begin{introthm} \label{nilp}
Let $N$ be a nontrivial, normal subgroup of a group $G$.  Assume that every character in $\irr {G \mid N}$ has central type. Then $G$ is nilpotent, and $\cl_G(g)=g[g,G]$ for every element $g\in N$. 
\end{introthm}

When we make the additional assumption that $N$ does not have prime power order, we obtain a much stronger conclusion.

\begin{introthm} \label{main part 2}
Let $N$ be a nontrivial, normal subgroup of a group $G$.  If $N$ does not have prime power order and every character in $\irr {G \mid N}$ has central type, then $G$ is a GVZ-group. 
\end{introthm}
 
When $N$ does have prime power order, we also obtain strong information on the structure of $G$.  Note that since $G$ is nilpotent by Theorem \ref{nilp} $Q$ will be normal in this theorem.

\begin{introthm} \label{main partial}
Let $p$ be a prime, and suppose $N$ is a nontrivial, normal subgroup of a group $G$ such that every character in $\irr {G \mid N}$ has central type. If $N$ is a $p$-group and $Q$ is the (normal) Hall $p$-complement of $G$, then $Q$ is a GVZ-group. 
\end{introthm}

In Theorem B of \cite{ourpre}, we show when $G$ is a GVZ-group that the nilpotence class of $G$ is bounded by $|\cd G|$, where $\cd G$ is the set of degrees of the irreducible characters of $G$.  We now obtain a similar result regarding the nilpotence class of $N$ when every character in $\irr {G \mid N}$ has central type.

\begin{introthm} \label{main part 3}
Let $N$ be a nontrivial, normal subgroup of a group $G$.  If every character in $\irr {G \mid N}$ has central type, then the nilpotence class of $N$ is at most $|\{ \chi (1) \mid \chi \in \irr {G \mid N} \}|$. 
\end{introthm}
 
As suggested above, a conjugacy class $\cl_G(g)$ satisfying $\cl_G(g)=g[g,G]$ can be considered a conjugacy class analog of a central type character. We call such a conjugacy class a {\it flat} class of $G$, or we simply say that the element $g$ is flat in $G$. We have seen that a group $G$ is nilpotent when each of its elements is flat.  We mention that flat elements and groups where all the elements are flat were first studied in \cite{flatness}.  We next present a generalization of this fact for normal subgroups suggested by Theorem~\ref{nilp}.

\begin{introthm}\label {flat 1}
Suppose every element in $N$ is flat in $G$. Then $N$ is contained in $Z_\infty(G)$, the hypercenter of $G$. In particular, $N$ is nilpotent.
\end{introthm}

Finally, we consider the pairs $(G,N)$ where every member of $\irr{G\mid N}$ is fully ramified over $N$. We will see that these groups have very special structure. Moreover, they provide examples of groups satisfying the premise of Theorem~\ref{nilp} that are not GVZ-groups.

\begin{introthm}\label{central cam}
Let $N$ be a normal subgroup of $G$. If each member of $\irr {G \mid N}$ is fully ramified over $N$, then $G$ is a $p$-group and $N = Z(G)$.
\end{introthm}

\section{Partial GVZ groups}

Suppose that $G$ has a nontrivial, normal subgroup $N$ so that each member of $\irr {G \mid N}$ has central type. In this situation, we say that $G$ is a {\it partial GVZ-group with respect to $N$}.  In this section, we prove that partial GVZ-groups are nilpotent groups that are GVZ-groups up to perhaps a single Sylow direct factor.  We begin with a lemma about the intersection of the kernels of the characters in $\irr {G \mid N}$.

\begin{lem}\label{trivial int}
Let $N$ be a nontrivial, normal subgroup of $G$. Then the intersection of the kernels of the characters in $\irr{G\mid N}$ is trivial.
\end{lem}

\begin{proof}
Write $K=\bigcap_{\chi\in\irr{G\mid N}}\ker(\chi)$. Let $\psi\in\irr{N}$ be nonprincipal. Notice that $K\le\ker(\chi)$ if $N\nleq\ker(\chi)$. In particular, $K\le \ker(\chi)$ for each $\chi\in\irr{G\mid \psi}$. It follows that $K\le\ker(\psi^G)$, which is contained in $N$. But $K\cap N$ is trivial, as this is the intersection of the kernels of all of the irreducible characters of $G$. It follows that $K=1$, as desired.
\end{proof}

This next lemma appeared as Lemma 3.8 in \cite{chain2}.

\begin{lem}\label{center int}
Let $N$ and $M$ be normal subgroups of $G$. Write $Z(G/N)=Z_N/N$ and $Z(G/M)=Z_M/M$. Then $Z(G/(N\cap M))=(Z_N\cap Z_M)/(N\cap M)$.
\end{lem}

We now come to the first main theorem of this section.  We consider the Sylow subgroups of partial GVZ groups.

\begin{thm}\label{partial gvz}
Let $N$ be a nontrivial, normal subgroup of $G$ and suppose that every character $\chi\in\irr{G\mid N}$ is a central type character. Let $S$ be a Sylow subgroup of $G$. Then $S\cap Z(G)=Z(S)$. If $\vartheta\in\irr{S}$ and $N\nleq\ker(\vartheta)$, then $\vartheta$ is a central type character. Furthermore, if $S$ intersects $N$ trivially, then $S$ is a GVZ-group. 
\end{thm}

\begin{proof}	
To prove the first statement, we essentially use the argument given in the proof of \cite[Theorem 2]{DJ}, taking a bit of extra care where necessary. Let $S$ be a Sylow subgroup of $G$, say for the prime $p$. Let $\vartheta\in\irr{S}$ satisfy $N\nleq\ker(\vartheta)$. Since $N\nleq\ker(\vartheta^G)$, there exists a character $\chi\in\irr{G\mid N}$ lying over $\vartheta$. In particular, $\chi$ is a central type character of $G$. Write $\chi_{Z(\chi)}=\chi(1)\lambda$, where $\lambda\in\irr{Z(\chi)}$. Let $R=Z(\chi)S$. Let $\gamma\in\irr{R\mid\vartheta}$ such that $\inner{\chi_R,\gamma}>0$. Note that $S\cap\ker(\chi)\le\ker(\vartheta)$, so $\vartheta$ can be considered a character of $\widebar{S}=S\ker(\chi)/\ker(\chi)$. Similarly $\gamma$ can be considered a character of $R/\ker(\chi)$. Let $\widebar{X}$ denote the image of $X\subseteq G$ under the projection $G\to G/\ker(\chi)$. Then $\widebar{R}=Z(\widebar{G})\widebar{S}=O_{p'}(Z(\widebar{G}))\times\widebar{S}$ and so $\gamma_{\widebar{S}}=\vartheta$. Since $\gamma$ lies over $\lambda$ and $\lambda^G$ is homogeneous, $\gamma^{\widebar{G}}=e\chi$ for some integer $e$. Then $e=\norm{G:R}\gamma(1)/\chi(1)=\norm{G:R}\vartheta(1)/\chi(1)$. Since $\norm{G:R}$ is not divisible by $p$, it follows that $\chi(1)^2_p=\norm{G:Z(\chi)}_p=\norm{S:S\cap Z(\chi)}=\norm{\widebar{S}:\widebar{S}\cap Z(\widebar{G})}$ divides $\vartheta(1)^2$. Since $\widebar{G}$ is a group of central type, $\widebar{S}\cap Z(\widebar{G}) = Z(\widebar{S})$ \cite[Theorem 2]{DJ}, from which it follows that $\vartheta$ is a central type character of $S$.
	
Now we show that $S\cap Z(G)=Z(S)$. Let $\chi\in\irr{G\mid N}$. As before, let $\widebar{X}$ denote the image of $X\subseteq G$ under the projection $G\to G/\ker(\chi)$. Since $\widebar{G}$ is a group of central type, we conclude from \cite[Theorem 2]{DJ} that $Z(\widebar{S})=\widebar{S}\cap Z(\widebar{G})$. By the modular law, we see that $Z(\widebar{S})=\widebar{S}\cap Z(\widebar{G})=\widebar{S\cap Z(\chi)}$. The isomorphism $\widebar{S}\to S/(S\cap\ker(\chi))$ therefore implies that $Z(S/(S\cap \ker(\chi))) = (S\cap Z(\chi))/(S\cap\ker(\chi))$. Writing $K=\bigcap_{\chi\in\irr{G\mid N}}\ker(\chi)$ and $Z=\bigcap_{\chi\in\irr{G\mid N}}Z(\chi)$, it follows from Lemma~\ref{center int} that $Z(S/(S\cap K))=(S\cap Z)/(S\cap K)$, as $\chi\in\irr{G\mid N}$ was chosen arbitrarily. Since $K=1$ by Lemma~\ref{trivial int}, we conclude from Lemma~\ref{center int} that $Z=Z(G)$. Thus $S\cap Z(G)=Z(S)$, as required.

Finally assume that $S\cap N=1$. Let $\psi\in\irr{N}$ be nonprincipal and let $\chi\in\irr{G\mid \psi}$. Then $\chi\in\irr{G\mid N}$, so $\chi$ is a central type character. Then $G/\ker(\chi)$ is a group of central type and $S\cap \ker(\chi)\le S\cap N=1$. Thus $S\ker(\chi)/\ker(\chi)\cong S$ is a group of central type \cite[Theorem 2]{DJ}. 
\end{proof}

With Theorem~\ref{partial gvz}, we will show $G$ must actually be nilpotent, which is the first statement in Theorem~\ref{nilp}. Before doing this however, we prove the second statement in Theorem~\ref{nilp}. In fact, we prove something slightly stronger.

\begin{thm}\label{flat}
Let $N$ be a normal subgroup of $G$. If every member of $\irr{G\mid [N,G]}$ has central type, then every element of $N$ is flat in $G$.
\end{thm}

\begin{proof}
Let $g\in N$. By column orthogonality
\[\norm{C_G(g)}=\sum_{\chi\in\irr{G\mid[g,G]}}\norm{\chi(g)}^2+\sum_{\chi\in\irr{G/[g,G]}}\norm{\chi(g)}^2.\]
Since $\irr{G/[g,G]}$ is exactly the set of irreducible characters $\chi$ for which $g\in Z(\chi)$, we see that the second sum is $\norm{G/[g,G]}$. Note that $[g,G]\le [N,G]$ and so $\irr{G\mid[g,G]}\subseteq\irr{G\mid [N,G]}$. Thus each member of $\irr{G\mid [g,G]}$ is a central type character. Also $g\notin Z(\chi)$ for any member of $\irr{G\mid [g,G]}$. This means the first sum is zero, which implies $\norm{C_G(g)}=\norm{G/[g,G]}$. We deduce that $\norm{\mathrm{cl}_G(g)}=\norm{[g,G]}$, and the result follows.
\end{proof}

Next, we provide an example that shows the converse of Theorem \ref{flat} is false.

\begin{exam}
The converse to Lemma~\ref{flat} is false. Let $G$ be nonnilpotent with $N = Z_2(G) > Z(G)$. Then $g[g,G] = \mathrm{cl}_G(g)$ for every $g\in N$, but by Theorem~\ref{nilp} there must exist a character in $\irr{G\mid [N,G]}$ that is not central type since $G$ is not nilpotent.
\end{exam}

Although the converse of Lemma~\ref{flat} is false, we can still determine that $N$ is nilpotent under the assumption that $\mathrm{cl}_G(g)=g[g,G]$ for all elements $g\in N$. In fact, $N$ is hypercentral in this case.
We will prove this in Theorem~\ref{flat 1}.

\begin{proof}[Proof of Theorem~\ref{flat 1}]
%

We work by induction on $\norm{G}$. Let $M\le N$ be minimal normal, and let $1\ne g\in M$.  We know $1 \notin \cl_G(g)$.  Since $g[g,G] = \cl_G(g)$, we see that $g^{-1} \notin [g,G]$, and thus, $g \notin [g,G]$.  Hence, $[g,G] < \inner{g}^G$ which must be all of $M$ by minimality. Thus, $[g,G] = 1$ and $g\in Z(G)$. Since $M$ is minimal normal, we conclude that $M\le Z(G)$. Observe that $G/M$ satisfies the hypotheses of the theorem, so the inductive hypothesis implies that $N/M\le Z_\infty(N/M)$. Since $M\le Z(G)$, this implies $N\le Z_\infty(G)$, as desired.
\end{proof}

\begin{rem}
Observe that if $N$ satisfies the premise of Theorem~\ref{flat 1}, then for each $g\in N$ and $\chi\in\irr{G\mid N}$, either $g\in Z(\chi_N)$ or $\chi(g) = 0$. We deduce from Clifford's Theorem that either $g\in Z(\psi)$ or $\sum_{x\in G}\psi^x(g) = 0$ for each $g\in N$ and $\psi\in\irr{N}$. In particular, every $G$-invariant irreducible character of $N$ has central type.
\end{rem}

Since flat elements stay flat after descending to a quotient, we can still guarantee the nilpotence of $N$ under a slightly weaker hypothesis.

\begin{cor}
	Suppose that every element in $N\setminus Z(N)$ is flat in $G$. Then $N$ is nilpotent.
\end{cor}

\begin{proof}
	Write $Z = Z(N)$, and let $\widebar{X}$ denote the image of $X \subseteq G$ under the projection $G\to G/Z$. Since $\mathrm{cl}_G (g) = g[g,G]$ for every $g\in N\setminus Z$, it follows that $\mathrm{cl}_{\widebar{G}} (\bar{g}) = \bar{g} [\bar{g},\widebar{G}]$ for every element $\bar{g} \in\widebar{N}$. Therefore, $\widebar{N}$ is nilpotent by Theorem~\ref{flat 1}, which clearly implies that $N$ is nilpotent.
\end{proof}

We now present the proof of Theorem~\ref{nilp}.

\begin{proof}[Proof of Theorem \ref{nilp}]
The second statement follows immediately from Theorem~\ref{flat}. To prove that $G$ is nilpotent, we work by induction on $\norm{G}$. By Theorem~\ref{partial gvz}, $Z(G)$ is the product the subgroups $Z(S)$ as $S$ ranges over all Sylow subgroups of $G$. In particular $Z(G)>1$. If $G$ is a $p$-group, we are done. So assume that $G$ is divisible by at least two distinct primes $p$ and $q$. Let $P$ be a Sylow $p$-subgroup of $G$, and let $Q$ be a Sylow $q$-subgroup of $G$. By Theorem~\ref{partial gvz}, $Z(P)$ and $Z(Q)$ are central and therefore normal in $G$. Without loss, we may assume that $N\nleq Z = Z(P)$.  Then $\irr{G/Z\mid NZ/Z}\subseteq\irr{G\mid N}$ is nonempty. By the inductive hypothesis, $G/Z$ is nilpotent. Since $Z\le Z(G)$, it follows that $G$ is nilpotent.
\end{proof}

We have seen that $G$ is nilpotent if all the characters in $\irr {G \mid N}$ have central type. It is interesting to ask if there are subsets of $G$ that imply $G$ is nilpotent, given each of its elements are flat in $G$. We have seen that this is not the case for an arbitrary normal subgroup of $G$, even if it is noncentral. It seems natural to ask if having all elements lying outside of $N$ being flat implies that $G$ is nilpotent. The answer is also no. As an example, let $G$ be a Frobenius group of order $pq$, where $p$ and $q$ are primes. In this case, every element lying outside of the Frobenius kernel will be flat, and $G$ is certainly not nilpotent.  

This next theorem includes Theorem \ref{main partial} from the Introduction.

\begin{thm} \label{nilp2}
Let $N$ be a normal subgroup of $G$. Then all characters in $\irr {G \mid N}$ have central type if and only if $G=S\times Q$, where $Q$ is a GVZ-group, $S$ is a Sylow subgroup of $G$ that intersects $N$ nontrivially, and all characters in $\irr{S\mid S\cap N}$ have central type.
\end{thm}

\begin{proof}
Suppose first that all of the characters in $\irr {G \mid N}$ are central type. Let $\{p_1,p_2,\dotsc,p_n\}$ be the set of prime divisors of $\norm{G}$, and assume that $\{p_1,p_2,\dotsc,p_\ell\}$ is the set of prime divisors of $\norm{N}$. Let $S_i$ be a Sylow $p_i$-subgroup of $G$ for each $i$. From Theorem \ref{nilp} we see that $G$ is nilpotent, so $G$ is the direct product of the $S_i$. Let $\chi\in\irr{G\mid N}$. Then there exist characters $\vartheta_i\in\irr{S_i}$ such that $\chi = \vartheta_1\times\vartheta_2\times\dotsb\times\vartheta_n$. Since $N\nleq\ker(\chi)$, there exists $j$, $1\le j\le \ell$, such that $S_j\cap N\nleq\ker(\vartheta_j)$. Let $\mu_i\in\irr{S_i}$ for each $1\le i\le n$, $i\ne j$. Write $\mu_j = \vartheta_j$ and let $\xi = \mu_1\times \mu_2\times\dotsb\times\mu_n$. Then $\xi\in\irr{G\mid N}$, so $\xi$ is a central type character. Thus $\prod_i\mu_i(1)^2 = \chi(1)^2 = \norm{G:Z(\chi)} = \prod_i\norm{S_i:Z(\mu_i)}$. By comparing $p_i$-parts for each $i$, we deduce that $\mu_i(1)^2 = \norm{S_i:Z(\mu_i)}$ for each $i$. Since each $\mu_i$, $i\ne j$, was chosen arbitrarily, it follows that each $S_i$ for $i\ne j$ is a GVZ-group. Thus $Q = \prod_{i\ne j}S_i$ is also a GVZ-group, and $G = S_j\times Q$. Finally, let $\vartheta\in\irr{S_j\mid S_j\cap N}$. Then $\vartheta$ lifts to a character $\tilde{\vartheta}\in\irr{G/Q\mid NQ/Q}$. Since $N\nleq\ker(\tilde{\vartheta})$, $\tilde{\vartheta}$ has central type. Thus $\vartheta(1)^2 = \tilde{\vartheta}(1)^2 = \norm{G:Z(\tilde{\vartheta})}=\norm{S_j:Z(\vartheta)}$, from which it follows that $\vartheta$ has central type. This completes the proof of the forward direction.

Now, suppose that $G = S\times Q$, $Q$ is a GVZ-group, $S$ is a Sylow subgroup for which $S\cap N > 1$, and all characters in $\irr{S\mid S\cap N}$ have central type.  Suppose $\chi \in \irr {G \mid N}$.  We can write $\chi = \vartheta \times \gamma$ where $\vartheta \in \irr S$ and $\gamma \in \irr Q$.  Notice that $\chi_N = \gamma (1) \vartheta_N$ and so $\vartheta \in \irr {S \mid S\cap N}$.   It is not difficult to see that $Z(\chi) = Z(\vartheta) \times Z(\gamma)$.  We know, by hypothesis, that $\vartheta$ is fully ramified with respect to $Z(\vartheta)$ and $\gamma$ is fully ramified with respect to $Z(\gamma)$.  It is not difficult to see that $\chi$ is fully ramified with respect to $Z(\chi)$.  
\end{proof}

As a corollary, we obtain Theorem \ref{main part 2} from the Introduction.

\begin{cor}
Let $N$ be a normal subgroup of $G$. If all characters in $\irr {G \mid N}$ have central type and $N$ is not a $p$-group for any prime $p$, then $G$ is a GVZ-group.
\end{cor}

\begin{proof}
We first appeal to Theorem \ref{nilp} to see that $G$ is nilpotent.  Since $N$ is not a $p$-group for any prime $p$,  we can find distinct primes $p_1$ and $p_2$ that divide $|N|$.  For $i = 1, 2$, take $N_i$ to be the Sylow $p_i$-subgroup of $N$.  Since $G$ is nilpotent, $N$ is nilpotent; so $N_i$ is normal and thus characteristic in $N$ and hence, they are normal in $G$.  Using Theorem \ref{nilp2} with $N_i$ as the normal subgroup, we see that all the Sylow subgroups of $G$ are GVZ-groups.  (We use $N_1$ to see that Sylow $p_2$-subgroup is a GVZ-group and $N_2$ for the Sylow $p_1$-subgroup.)  Now, we know that $G$ is nilpotent and that all the Sylow subgroups are GVZ-groups, and so, we can conclude that $G$ is a GVZ-group.
\end{proof}

We now present the proof of Theorem~\ref{main part 3}, which is an adaptation of the usual Taketa argument. Our proof is nearly identical to our proof of Theorem B of \cite{ourpre} with just a few subtle differences.  When $G$ is a nilpotent group, we write $c(G)$ for the nilpotence class of $G$, and we set $G_1 = G$ and $G_{i+1} = [G_i,G]$ for $i \ge 1$ for the terms of the lower central series for $G$.  Following \cite{Knutson}, we set $\cd{G\mid N}=\{\chi(1) \mid \chi\in\irr{G\mid N}\}$. Using this notation, we restate Theorem \ref{main part 3} in a slightly different form.

\begin{thm}
Let $N$ be a nontrivial, normal subgroup of $G$.  If all of the characters in $\irr {G \mid N}$ have central type, then $c (N) \le |\cd {G \mid N}|$.
\end{thm}

\begin{proof}
We work by induction on $|G|$.  If $N$ is abelian, then the result is trivial.  Thus, we may assume that $N$ is not abelian, which implies that $G$ is not abelian.  Let $d_1 < \dots < d_n$ be the distinct degrees in ${\rm cd} (G \mid N)$.  Consider a character $\chi \in \irr {G \mid N}$.  If $\ker (\chi) > 1$, then $|G/\ker (\chi)| < |G|$, and ${\rm cd} (G/\ker (\chi) \mid N\ker (\chi)/\ker (\chi)) \subseteq {\rm cd} (G \mid N)$.  By induction, we have $N_n \le \ker (\chi)$.  Thus, if $G$ does not have a faithful character in $\irr {G \mid N}$, then $N_n \le \cap_{\chi \in \irr {G \mid N}} \ker (\chi) = 1$ by Lemma \ref{trivial int}.  Therefore, we may assume that there exists $\chi \in \irr {G \mid N}$ with $\ker (\chi) = 1$.  This implies that $Z (\chi) = Z(G)$.  We have $d_i^2 \le |G:Z(G)| =\chi (1)^2$ for every integer $i$ with $1 \le i \le n$.  Thus, $\chi (1) = d_n$.  Notice that if $a \in {\rm cd} (G/Z(G) \mid NZ(G)/Z(G))$, then $a^2 < |G:Z(G)| = d_n^2$.  It follows that $|{\rm cd} (G/Z(G) \mid NZ(G)/Z(G))| \le n -1$.  By the inductive hypothesis, we have that $N_{n-1} \le Z(G)$.  It follows that $N_{n-1} \le Z(G) \cap N \le Z(N)$.  This implies that $c (N) \le n$ as desired.
\end{proof}

In Theorem 6.3 of \cite{gagola1}, Gagola has proven that if $Q$ is any $p$-group, then there exists a $p$-group $P$ so that $Q$ is isomorphic to a subgroup of $P/Z(P)$ and every character in $\irr {P \mid Z(P)}$ is fully ramified with respect to $Z(P)$.  Notice that this will imply that every irreducible character in $\irr {P \mid Z(P)}$ will have central type and $\cd {P \mid Z(P)} = \{ |P:Z(P)|^{1/2} \}$, so $|\cd { P \mid Z(P)}| = 1$.  Since $Q$ is abitrary, this implies that we have no hope of bounding $c(G)$ in terms of $|\cd {G \mid N}|$ when every character in $\irr {G \mid N}$ has central type.

Even if we assume that $G$ is a GVZ-group, we cannot hope to bound the nilpotence class of $G$ in terms of $|\cd {G \mid N}|$ for a subgroup $N$.  To see this, consider the groups found in Example 3 in \cite{ML19gvz}.  The groups in that example are all GVZ-groups and all satisfy that $|\cd {G \mid Z(G)}| = 1$, but can be chosen to have whatever nilpotence class desired.

\section{The subgroup $R(G)$}

Suppose that $G$ is a partial GVZ-group with respect to each of the normal subgroups $N$ and $M$. It is not difficult to see that $\irr {G \mid NM} = \irr {G \mid N} \cup \irr {G \mid M}$, so $G$ is a partial GVZ-group with respect to $NM$. In particular, there is a unique largest normal subgroup $R(G)$ such that every member of $\irr {G \mid R(G)}$ has central type. In this section, we discuss some properties of this subgroup.

\begin{lem}
Let $G$ be a group.  The following statements are true:
\begin{enumerate}[label={\rm(\arabic*)}]
\item $R(G) = G$ if and only if $G$ is a GVZ-group;
\item If $R(G) > 1$, then $G = S \times Q$, where $Q$ is a GVZ-group and $S$ is either trivial or a Sylow $p$-subgroup satisfying $1 < R(S) = R(G) < S$ for some prime $p$;
\item $R(G/R(G)) = 1$.
\end{enumerate}
\end{lem}

\begin{proof}
Statement (1) is obvious from the definition, and statement (2) follows immediately from Theorem~\ref{main part 2} and Theorem~\ref{main partial}. So we only show (3). If $R(G) = 1$, then (3) is clear. So assume that $R(G) > 1$, and write $L/R(G) = R(G/R(G))$. Let $\chi \in \irr {G \mid L}$. If $R(G) \le \ker(\chi)$, then $\chi$ has central type since $L/R(G) = R (G/R(G))$. If $R(G) \nleq \ker(\chi)$, then $\chi$ has central type by the definition of $R(G)$. Thus $L\le R(G)$, from which it follows that $L = R(G)$.
\end{proof}

We next look at $R(G)$ when $G$ is a direct product.

\begin{lem}\label{direct products}
Let $M$ and $N$ be groups. Assume that $M$ is not a GVZ-group. The following statements are true:
\begin{enumerate}[label={\rm(\arabic*)}]
\item If $N$ is not a GVZ-group, then $R (M \times N) = 1$;
\item If $N$ is a GVZ-group, then $R (M \times N) = R(M)$.
\end{enumerate}
\end{lem}

\begin{proof}
We prove the contrapositive of statement (1). Write $R = R (M \times N)$. Assume that $R > 1$. We want to show that that $N$ is a GVZ-group.  Since the result is immediate if $N$ is abelian, we may assume that $N$ is nonabelian. Fix the character $\nu \in \irr{N}$, so that $\nu$ is nonlinear. We claim that there exists a character $\mu \in \irr {M}$ so that $\chi = \mu \times \nu \in \irr {M \times N \mid R}$. To see this, assume on the contrary that $R \le \ker(\mu \times \nu)$ for every character $\mu \in \irr {M}$. Consider an element $mn\in R$, where $m\in M$ and $n\in N$. Then $mn \in \ker(1_M \times \nu) = M \times \ker (\nu)$, so $n \in \ker(\nu)$. So $\mu(m) \nu(n) = \mu(m)\nu(1) = \mu(1)\nu(1)$ for every character  $\mu \in \irr{M}$, from which it follows that $m \in \bigcap_{\mu \in \irr{M}} \ker (\mu) = 1$. Thus we see that $R \le \ker(\nu) \le N$. 

Since $M$ is not a GVZ-group, we see that $M$ is not abelian. In particular, $\irr {M \mid M'}$ is nonempty. Fix a character $\mu \in \irr {M \mid M'}$ and consider a character $\theta \in \irr {N \mid R}$. We obtain $\chi = \mu \times \theta \in \irr {M \times N \mid R}$.  Fix an element $m \in M \setminus Z(\mu)$.  This implies $m \notin Z(\chi)$, so $0 = \chi(m) = \mu(m)\theta(1)$ and it follows that $\mu (m) = 0$.  We deduce that $\mu$ has central type, and since $\mu$ was arbitrary, we conclude that $M$ is a GVZ-group, which is a contradiction. Thus, we may find a character $\mu \in \irr {M}$ so that $\chi = \mu \times \nu \in \irr {M \times N \mid R}$, as claimed. 

Consider elements $m \in Z(\mu)$ and $n \in N \setminus Z(\nu)$. Since $Z (\chi) = Z (\mu) \times Z (\nu)$, we determine that $mn \notin Z(\chi)$. Since $\chi$ has central type, we calculate $0 = \chi (mn) = \mu(m)\nu(n)$. Because $\mu(m)\ne 0$, we must have $\nu(n) = 0$. Thus, $\nu$ has central type and since $\nu \in \irr {N \mid N'}$ was chosen arbitrarily, $N$ is a GVZ-group. 
	
Now we show conclusion (2). Assume that $N$ is a GVZ-group. Let $R = R (M \times N)$. We prove that $R \cap N$ must be trivial.  Suppose $R \cap N > 1$.  We claim that this will imply that $M$ must be a GVZ-group which is contradiction.  As we have seen, if $M$ is abelian, then $M$ will be a GVZ-group, so we assume that $M$ is not abelian.  Since $R \cap N > 1$, we can find a character $1 \ne \gamma \in \irr {R \mid N}$.  Consider characters $\nu \in \irr {N \mid \gamma}$ and $\mu \in \irr {M \mid M'}$.  Since $\gamma \ne 1$, we see that $\chi = \nu \times \mu \in \irr {N \times M \mid R(G)}$.   for if not the above argument shows that  We claim that $R\cap M$ is nontrivial. Assume, on the contrary, that $R\cap M$ is trivial. Let $m\in M$, $n\in N$ such that $mn\in R$. Since $R\lhd G$, it follows that $[m,M][n,N]\subseteq (M\cap R)(N\cap R)=1$. So $m\in Z(M)$ and $n\in Z(N)$. In particular $mn\in Z(G)$ and $\inner{mn}$ is a normal subgroup of $G$. Since $mn\in R$, $\chi\in\irr{M\times N\mid\inner{mn}}$ has central type. Let $\mu\in\irr{M\mid \inner{m}}$. Then $mn\notin\ker(\mu\times 1_N)$, so $\mu\times 1_N$ has central type. But this implies that $\mu$ has central type. So $m\in R(M)$. Since $m\in Z(M)\le Z(G)$, $\inner{m}$ is also a normal subgroup of $G$. Since $N$ is a GVZ-group and $\inner{m}\le R(M)$, every member of $\irr{M\times N\mid\inner{m}}$ has central type, which contradicts the fact that $R\cap M=1$. So $R\cap M > 1$, as claimed. Since $R\cap N = 1$, this implies that $R\le M$. So $\mu\times\nu\in\irr{M\times N\mid R}$ if and only if $\mu\in\irr{M\mid R}$. Since $N$ is a GVZ-group, $\mu\times\nu$ has central type if and only if $\mu$ does. From this it follows that $R(M) \le R$, since every member of $\irr{M\times N\mid R(M)}$ must have central type, and that $R\le R(M)$, since every member of $\irr{M\mid R}$ must have central type. Thus $R=R(M)$, as desired.
\end{proof}

Recall that the {\it vanishing-off subgroup} $V(\chi)$ of a character $\chi$ is the smallest normal subgroup $V$ such that $\chi$ vanishes on $G\setminus V$ (e.g. see \cite[Chapter 12]{MI76}). It is not difficult to see that $V(\chi)=\inner{g\in G\mid \chi(g)\ne 0}$ and that $Z(\chi)\le V(\chi)$. Since central type characters vanish off their centers, these are exactly the characters whose vanishing-off subgroup coincide with their center.

\begin{lem}\label{R intersection}
Let $\mathcal{V}=\{\chi\in\irr{G}\mid V(\chi) > Z(\chi)\}$. Then $R(G)=\bigcap_{\chi\in \mathcal{V}}\ker(\chi)$.
\end{lem}

\begin{proof}
Let $W=\bigcap_{\chi\in \mathcal{V}}\ker(\chi)$. If $\chi\in\irr{G\mid W}$, then $V(\chi) = Z(\chi)$ and therefore has central type. So $W\le R(G)$. 
	
Since every member of $\irr{G\mid R(G)}$ satisfies $V(\chi) = Z(\chi)$, $\mathcal{V}\subseteq \irr{G/R(G)}$. So $R(G)\le W$ as well, which completes the proof.
\end{proof}

As a corollary, we see that every member of $\irr{G\mid R(G)}$ is nonlinear, unless $G$ is a GVZ-group.

\begin{lem}
If $G$ is not a GVZ-group, then $R(G) < G'$.
\end{lem}

\begin{proof}
Let $\mathcal{V}=\{\chi\in\irr{G}\mid V(\chi) > Z(\chi)\}$. Since $G$ is not a GVZ-group, $\mathcal{V}$ is not empty. Let $\chi\in\mathcal{V}$. For each $\lambda\in\irr{G\mid G'}$, we have $\chi\lambda\in\mathcal{V}$. Since $R(G)=\bigcap_{\xi\in\mathcal{V}}\ker(\xi)$ by Lemma~\ref{R intersection}, we see that $R(G)\le \ker(\chi\lambda)\cap\ker(\chi) \le \ker(\lambda)$ for each $\lambda\in\irr{G\mid G'}$. Thus $R(G)\le G'$.  Notice that if $R (G) = G'$, then every nonlinear irreducible character of $G$ would be central type and that would imply that $G$ is a GVZ-group, a contradiction.  Thus, we must have $R (G) < G'$.
\end{proof}

We next present a result that allows us to find examples of partial GVZ-groups that are not GVZ-groups. To do this, we consider a characteristic subgroup that is defined in Section 6 of \cite{chain2}. To define it, we begin by defining $U(G)$ to be the largest normal subgroup $U$ of $G$ such that every member of $\irr{G\mid U}$ is fully ramified over $Z(G)$. By defining $U_2(G)/U(G)=U(G/U(G))$, etc, we define an ascending series of subgroups $U_i(G)$ and let $U_\infty(G)$ denote its terminal member.

\begin{lem}\label{Uinfinity}
Let $N=U_\infty(G)$. Then every character in $\irr{G\mid N}$ has central type. In particular, $N\le R(G)$.
\end{lem}

\begin{proof}
We work by induction on $\norm{G}$. Let $\chi\in\irr{G\mid N}$. Suppose that $N > U(G)$. By the inductive hypothesis, every character $\chi\in\irr{G/U\mid N/U}$ has central type. Thus we may assume that $\chi\in\irr{G\mid U}$. By definition, $V(\chi) = Z(G)$ and so $\chi$ vanishes on $G\setminus Z(G)$. In particular, $\chi$ has central type. Thus every character in $\irr{G\mid N}$ has central type, as required.
\end{proof}

As a consequence of Lemma~\ref{Uinfinity}, a non GVZ-group $G$ with $U_\infty(G)>1$ provides a nontrivial example of a partial GVZ-group. In particular, if $G$ is not a GVZ-group and $U(G)>1$, then $G$ is a partial GVZ-group with respect to $U(G)$. The smallest $2$-group with this property has order $128$. One may readily check that if $G$ is the group \verb+SmallGroup (128, 71)+ from MAGMA's Small Groups Library, then $1 < U(G) = R(G)$. Moreover, one may even find groups $G$ where $1<U(G) < R(G) < G$. Indeed the group \verb+SmallGroup (256, 6442)+ is an example of such a group.

A group that satisfies $U_\infty(G)=G$ is necessarily a GVZ-group, but the converse is not true (see \cite[Theorem 6.15]{chain2}). Given Lemma~\ref{Uinfinity}, one may be tempted to guess that $R(G) = U_\infty(G)$ if $R(G) > 1$ but $G$ is not a GVZ-group. As an example, if $G$ is the group \verb+SmallGroup(512, 51833)+, then $U(G) < U_\infty(G) = R(G) < G$. However, this is always not the case. For example, we have seen that if $M$ is the group \verb+SmallGroup (128, 71)+, then $1 < U(M) = R(M)$. Let $N$ be a nonabelian GVZ-group and let $G = M \times N$. By Lemma~\ref{direct products}, $1 < R(M) = R(G) < G$. So $G$ is not a GVZ-group. Moreover $U(G) = 1$ \cite[Lemma 6.7]{chain2}. This gives an example of a group $G$ that satisfies $1 = U(G) < R(G) < G$. As another example, if $G$ is the group \verb+SmallGroup (256, 19628)+, then $G$ is not a GVZ-group, $U(G) = 1$, and $\norm{R(G)}=2$.

\section{Central Camina pairs}

We conclude with a special example of a partial GVZ-group. Recall that the subgroup $U(G)$ is defined in \cite{chain2} to be the largest normal subgroup $U$ of $G$ such that every character $\chi\in\irr{G\mid U}$ is fully ramified over $Z(G)$. It is easy to see that $U(G)\le Z(G)$ and that $G$ is a partial GVZ-group with respect to $U(G)$ when $U(G)>1$. In the case that $U(G)=Z(G)$, the pair $(G,Z(G))$ is an example of a {\it Camina pair}.
 
The study of Camina pairs began in \cite{camina} where Camina was looking at conditions that generalize both Frobenius groups and extra-special $p$-groups.  If $G$ is a group and $N$ is a normal subgroup of $G$, we say $(G,N)$ is a {\it Camina pair} if every element $g \in G \setminus N$ is conjugate to the entire coset $gN$.  There are a number of equivalent formulations of this condition.  For this paper, the most relevant is that $(G,N)$ is a Camina pair if and only if each nonprincipal irreducible character of $N$ induces homogeneously to $G$. 
 
Suppose that $(G,N)$ is a Camina pair, and let $\chi\in\irr{G\mid N}$. Then for any irreducible constituent $\psi$ of $\chi_N$, $\psi^G=e\chi$ for some integer $e$ called the {\it ramification index}. The integer $e$ satisfies $1\le e^2\le\norm{G:N}$. The condition $e=1$ is equivalent to $\psi^G\in\irr{G}$, and the condition $e^2=\norm{G:N}$ is equivalent to $\chi$ being fully ramified over $N$. It is well-known that every nonprincipal irreducible character of a subgroup $N$ induces irreducibly to $G$ if and only if $G$ is a Frobenius group and $N=F(G)$. Thus a Frobenius group is a Camina pair $(G,N)$ where the ramification index of each $\chi\in\irr{G\mid N}$ is as small as possible. Observe that Theorem~\ref{central cam} provides a characterization of those Camina pairs $(G,N)$ where the ramification index of each $\chi\in\irr{G\mid N}$ is as large as possible. 
  
Before proving Theorem~\ref{central cam}, we motivate the conclusion a bit. Let $(G,N)$ be a Camina pair. Then it is easy to see that $Z(G) \le N$, and we say $(G,N)$ is a {\it central Camina pair} in the extreme case that $N = Z(G)$.  Central Camina pairs are the focus of \cite{central camina}.  Note that $(G,Z(G))$ is a Camina pair if and only if $U(G)=Z(G)$, which happens if and only if $(G,U(G))$ is a Camina pair. It is known that if $(G,N)$ is a central Camina pair, then $G$ is a $p$-group and every character in $\irr {G \mid N}$ is fully ramified over $N$.  Thus Theorem~\ref{central cam} provides a sort of converse to this statement.

\begin{proof}[Proof of Theorem~\ref{central cam}]
We prove this by induction on $\norm {G}$. First, we show that $G$ must be a $p$-group.  Since every character in $\irr {G \mid N}$ vanishes on $G \setminus N$, we see that $(G,N)$ is a Camina pair, and that $G$ cannot be a Frobenius group with Frobenius kernel $N$. So either $N$ is a $p$-group, or $G/N$ is a $p$-group \cite[Theorem 2.2]{camina}. If $N$ is a $p$-group, then it follows from \cite[Theorem 5.1]{hominduction1} that $I_G (\psi)$ is a $p$-group for each nonprincipal $\psi \in \irr N$. Since every nonprincipal irreducible character of $N$ is fully ramified with respect to $G/N$, this would force $G$ to be a $p$-group as well. So we assume that $G/N$ is a $p$-group. 
	
Since $\norm {G:N}$ divides $\chi (1)$ for every character $\chi \in \irr {G \mid N}$, we know that $p$ divides $\chi (1)$ for every character $\chi \in \irr {G \mid N'}$ if $N' > 1$. It follows from \cite[Theorem D]{Knutson} that $N$ (and hence also $G$) is solvable. Since every character $\psi \in \irr N$ is invariant in $G$, we see that $[N,G] = N'$. Suppose that $1 < M \le N'$ is a normal subgroup of $G$.  Then $(G/M,N/M)$ is a Camina pair by the inductive hypothesis.  Since every nonprincipal character $\psi \in \irr {N/M}$ is fully ramified with respect to $G/N$, it follows from the inductive hypothesis that $N/M = Z (G/M)$. Thus, $N' = [N,G] \le M$, and we deduce that $N'$ is a minimal normal subgroup of $G$. Therefore, $N'$ is an elementary abelian $q$-group for some prime $q$.  Also, note that $N/N'$ is an elementary abelian $p$-group \cite[Theorem 2.2]{IDMac81}. If $q = p$, then we are done, so assume $q \ne p$. Then $N'$ is a normal $p$-complement for $G$, and $P \cong G/N'$. Since $Z (G/N') = N/N'$, we have $\norm {Z(P)} = \norm {N:N'} = \norm {N \cap P}$. But $Z (P) \le N$ by \cite[Proposition 3.4]{genFrobGps1}, which allows us to conclude $Z (P) = P \cap N$. By \cite[Theorem 3]{genFrobGps2}, $G$ must be a $p$-group, a contradiction. Therefore $q = p$ and $G$ is a $p$-group, as claimed. 
	
Since $N'$ is a minimal normal subgroup of $G$, it follows that $N' \le Z(G)$. By applying the Three-subgroups lemma, we deduce that $[G',N] = 1$. Since $N \le G'$, this implies $N \le Z (G')$, which in turn implies that $N$ is abelian. In particular, this yields $1 = N' = [N,G]$, from which we conclude $N \le Z(G)$. Since $(G,N)$ is a Camina pair, we also have $Z(G) \le N$, which completes the proof.
\end{proof}

Of course, we get the desired characterization of central Camina pairs as a corollary. Observe the similarity to the aforementioned characterization of Frobenius groups.

\begin{cor}\label{central cam 2}
The group $G$ is a central Camina pair if and only if $G$ has a normal subgroup $N$ such that every character $\chi \in \irr {G\mid N}$ is fully ramified over $N$. In this case, $N=Z(G)$.
\end{cor}

\end{document}